\documentclass[12pt]{article}

\usepackage{amsmath,amsthm,amsfonts,amssymb,latexsym,amscd,comment}

\begin{document}

\newtheorem*{theo}{Theorem}
\newtheorem*{pro}{Proposition}
\newtheorem*{cor}{Corollary}
\newtheorem*{lem}{Lemma}
\newtheorem{theorem}{Theorem}[section]
\newtheorem{corollary}[theorem]{Corollary}
\newtheorem{lemma}[theorem]{Lemma}
\newtheorem{proposition}[theorem]{Proposition}
\newtheorem{conjecture}[theorem]{Conjecture}
\newtheorem{commento}[theorem]{Comment}
\newtheorem{definition}[theorem]{Definition}
\newtheorem{problem}[theorem]{Problem}
\newtheorem{remark}[theorem]{Remark}
\newtheorem{example}[theorem]{Example}
\newcommand{\Naturali}{{\mathbb{N}}}
\newcommand{\R}{{\mathbb{R}}}
\newcommand{\Toro}{{\mathbb{T}}}
\newcommand{\Relativi}{{\mathbb{Z}}}
\newcommand{\HH}{\mathfrak H}
\newcommand{\KK}{\mathfrak K}
\newcommand{\LL}{\mathfrak L}
\newcommand{\as}{\ast_{\sigma}}
\newcommand{\tn}{\vert\hspace{-.3mm}\vert\hspace{-.3mm}\vert}
\def\mA{{\mathfrak A}}
\def\A{{\mathcal A}}
\def\mB{{\mathfrak B}}
\def\B{{\mathcal B}}
\def\C{{\mathcal C}}
\def\D{{\mathcal D}}
\def\F{{\mathcal F}}
\def\H{{\mathcal H}}
\def\J{{\mathcal J}}
\def\K{{\mathcal K}}
\def\L{{\cal L}}
\def\N{{\cal N}}
\def\M{{\cal M}}
\def\O{{\mathcal O}}
\def\P{{\cal P}}
\def\SS{{\cal S}}
\def\T{{\cal T}}
\def\U{{\cal U}}
\def\W{{\cal W}}
\def\Z{{\mathbb Z}}
\def\cc{{\mathbb C}}
\def\b{\lambda_B(P}
\def\j{\lambda_J(P}
\def\span{\operatorname{span}}
\def\Ad{\operatorname{Ad}}
\def\sg{\operatorname{Sym}}
\def\tr{\operatorname{tr}}
\def\mnk{{M_n(\cc)^{\otimes k}}}
\def\id{\operatorname{id}}
\def\en{\operatorname{End}}
\def\aut{\operatorname{Aut}}
\def\la{\langle}
\def\ra{\rangle}
\def\Fix{\bf{F}}

\title{On Invariant MASAs for Endomorphisms of the Cuntz Algebras}

\author{Jeong Hee Hong\footnote{This work was supported by National Research Foundation of Korea
Grant funded by the Korean Government (KRF--2008--313-C00039).}, Adam Skalski and
Wojciech Szyma{\'n}ski\footnote{This work was partially supported by the FNU Rammebevilling
grant `Operator algebras and applications' (2009--2011).}}

\date{10 January 2010}
\maketitle

\renewcommand{\sectionmark}[1]{}

\vspace{7mm}
\begin{abstract}
The problem of existence of standard (i.e.\ product-type) invariant MASAs
for endomorphisms of the Cuntz algebra $\O_n$ is
studied. In particular endomorphisms which preserve the canonical diagonal
MASA $\D_n$ are investigated. Conditions on a unitary
$w\in\U(\O_n)$ equivalent to the fact that the corresponding endomorphism
$\lambda_w$ preserves $\D_n$ are found, and it is shown that they may be satisfied
by unitaries which do not normalize $\D_n$. Unitaries giving rise to endomorphisms which leave all
standard MASAs invariant and have identical actions on them are
characterized. Finally some properties of examples of
finite-index endomorphisms of $\O_n$ given by Izumi and related
to sector theory are discussed  and it is shown that they lead to
an endomorphism of $\O_2$ associated to a matrix unitary which does not preserve any standard MASA.
\end{abstract}

\vfill\noindent {\bf MSC 2010}: 46L55, 46L40, 46L05

\vspace{3mm}

\newpage

\section{Introduction}

Systematic investigations of endomorphisms of the $C^*$-algebras $\O_n$ were initiated by Cuntz in \cite{Cun2}, where he noticed
the bijective correspondence between unitaries in $\O_n$ and unital endomorphisms of $\O_n$. From this, he derived a number of
interesting properties of such endomorphisms.

Thirty years later, endomorphisms of $\O_n$ still constitute an active area of research offering a number of challenging problems
and interesting connections with other fields, including index theory, noncommutative entropy, quantum groups and classical
dynamical systems (see for example \cite{CF,CP,Cun3,I,L}). In particular, significant progress has been achieved recently in the
study of endomorphisms corresponding to permutation unitaries, \cite{CKS,CS, CS2,K1,K2,SZ,S}.

In \cite{Cun2}, Cuntz considered also the following two problems: given a unitary $u\in\O_n$,
under what conditions the corresponding endomorphism $\lambda_u$ globally preserves
(i) the canonical UHF-subalgebra $\F_n$, and (ii) the canonical diagonal MASA $\D_n$?
He gave the complete answers in the case when $\lambda_u$ is an {\em automorphism} of
$\O_n$ and provided partial information in the general case when $\lambda_u$ is not
necessarily surjective. The former problem has been recently successfully tackled in \cite{CRS}. In the
present note, we consider some issues closely related to the latter.

The main focus of the present paper are MASAs of $\O_n$ globally invariant under the action of a unital endomorphism.  The best
studied and most important endomorphism of $\O_n$, the canonical shift $\varphi$, is often viewed as the noncommutative
generalization of the standard Bernoulli-Markov shift of the topological dynamics, as the restriction of $\varphi$ to the diagonal
is indeed the map induced by the standard shift. As any standard MASA in $\O_n$  (i.e.\ the image of the canonical diagonal MASA
$\D_n$ under a Bogolyubov automorphism) is naturally isomorphic to the algebra of continuous functions on the full
shift on $n$-letters $\mathfrak{C}$, the usual arena of symbolic dynamics (\cite{Kit}), it seems natural to ask what other classical
transformations of $\mathfrak{C}$ one can obtain by restricting endomorphisms of $\O_n$ to invariant standard MASAs. It is
easy to notice that the canonical shift leaves all standard MASAs invariant and always leads to the same transformation of
$\mathfrak{C}$. The examples in \cite{SZ} showed that there exist endomorphisms of $\O_n$ which reduce to the usual shift on
$\mathfrak{C}$ in some invariant standard MASA, but lead to a different transformation in another. One could rephrase this in the
following statement -- a classical system arising in symbolic dynamics may have different `quantizations' to endomorphisms of
Cuntz algebras and  at the same time an endomorphism of $\O_n$ may be a quantization of several different classical dynamical
systems. A specific motivation for considering such questions comes from usefulness in the context of noncommutative entropy
calculations of examining several different MASAs invariant under the same endomorphism (for example, see \cite{SZ,Sk2}).

In this work we initiate a systematic study of the topics listed above. In particular, for a unital endomorphism $\rho$ of $\O_n$
we consider the question of how to identify all its $\rho$-invariant standard (product-type) MASAs.  In Section 3 we show that this problem reduces to
deciding which unitaries $w$ in $\O_n$ have the property that the corresponding endomorphism $\lambda_w$ globally preserves the
diagonal $\D_n$. As already noted by Cuntz in \cite{Cun2}, this is certainly the case when $w$ normalizes $\D_n$. In the main part of the present
paper we consider this problem from a much more general perspective. In particular, our Theorem \ref{finitecond} gives a
necessary and sufficient condition for a unitary $w$ in the algebraic part of $\F_n$ so that $\lambda_w(\D_n)\subseteq\D_n$.
Verification of this condition involves only finitely many linear operations and could be easily performed by a computer.
Corollaries to Theorem \ref{finitecond} offer further simplifications of the general conditions under additional assumptions and
lead to  a large class of examples of such endomorphisms corresponding to unitary matrices not belonging to the normalizer of
$\D_n$.

In Section \ref{bogolyubov}, we characterize the unitaries which are associated with the endomorphisms of $\O_n$ which leave all
standard MASAs invariant and have identical actions on each of them. This is closely related to the work of Bratteli and Evans on
the canonical $\U(n)$ action on $\O_n$ in \cite{BE}, and earlier paper of Price on the action of $U(n)$ on the UHF algebra $\F_n$,
\cite{Price}. We note that commutation with Bogolyubov automorphisms has been exploited in the literature also in other contexts, see for
example Section 4 in \cite{I}.

Section \ref{izumi} is devoted to the analysis of examples of endomorphisms associated to finite abelian groups by Izumi in
\cite{I}. It turns out that Izumi's examples (arising from the subfactor theory), although not permutation endomorphisms
themselves, are at the same time square roots of a permutation endomorphism and compositions of a permutation endomorphism with a
Bogolyubov automorphism. In particular we obtain an endomorphism of $\O_2$ corresponding to a matrix unitary which does not
globally preserve any standard MASA.

\vspace{3mm}\noindent
{\bf Acknowledgements.} The first named author would like to thank Uffe Haagerup and other
members of mathematics department at Odense for their warm hospitality during her stay there
in the summer of 2009, when part of this work was done. The second named author's contribution to
this note was made during his visit to University of Tokyo in October-November 2009, funded by
a JSPS Short Term Postdoctoral Fellowship.

\section{Notation and preliminaries}

If $n$ is an integer greater than 1, then the Cuntz algebra $\O_n$ is a unital, simple,
purely infinite $C^*$-algebra generated by $n$ isometries $S_1, \ldots, S_n$, satisfying
$\sum_{i=1}^n S_i S_i^* = I$, \cite{Cun1}.
We denote by $W_n^k$ the set of $k$-tuples $\alpha = (\alpha_1,\ldots,\alpha_k)$
with $\alpha_m \in \{1,\ldots,n\}$, and by $W_n$ the union $\cup_{k=0}^\infty W_n^k$,
where $W_n^0 = \{0\}$. We call elements of $W_n$ multi-indices.
If $\alpha \in W_n^k$ then $|\alpha| = k$ is the length of $\alpha$.
If $\alpha = (\alpha_1,\ldots,\alpha_k) \in W_n$, then $S_\alpha = S_{\alpha_1} \ldots S_{\alpha_k}$
($S_0 = I$ by convention) is an isometry with range projection $P_\alpha=S_\alpha S_\alpha^*$.
Every word in $\{S_i, S_i^* \ | \ i = 1,\ldots,n\}$ can be uniquely expressed as
$S_\alpha S_\beta^*$, for $\alpha, \beta \in W_n$ \cite[Lemma 1.3]{Cun1}.

We denote by $\F_n^k$ the $C^*$-subalgebra of $\O_n$ spanned by all words of the form
$S_\alpha S_\beta^*$, $\alpha, \beta \in W_n^k$, which is isomorphic to the
matrix algebra $M_{n^k}({\mathbb C})$. The norm closure $\F_n$ of
$\cup_{k=0}^\infty \F_n^k$, is the UHF-algebra of type $n^\infty$,
called the core UHF-subalgebra of $\O_n$, \cite{Cun1}. It is the fixed point algebra
for the gauge action of the circle group $\gamma:U(1)\rightarrow{\rm Aut}(\O_n)$ defined
on generators as $\gamma_t(S_i)=tS_i$. For $k\in\Z$, we denote by $\O_n^{(k)}
:=\{x\in\O_n:\gamma_t(x)=t^k x\}$,
the spectral subspace for this action. In particular, $\F_n=\O_n^{(0)}$.
The $C^*$-subalgebra of $\F_n$ generated by
projections $P_\alpha$, $\alpha\in W_n$, is a MASA (maximal abelian subalgebra) both in $\F_n$ and in $\O_n$. We call it
the diagonal and denote $\D_n$. We also set $\D_n^k:=\D_n\cap\F_n^k$.
Throughout this paper we are interested in the inclusions
$$ \D_n \subseteq \F_n\subseteq \O_n.$$

We denote by $\SS_n$ the group of those unitaries in $\O_n$ which can be written
as finite sums of words, i.e., in the form $u = \sum_{j=1}^m S_{\alpha_j}S_{\beta_j}^*$
for some $\alpha_j, \beta_j \in W_n$. It turns out that $\SS_n$ is isomorphic to
the Higman-Thompson group $G_{n,1}$ \cite{B,Ne}. We also denote $\P_n=\SS_n\cap\U(\F_n)$. Then
$\P_n=\cup_k\P_n^k$, where $\P_n^k$ are permutation unitaries in $\U(\F_n^k)$.
That is, for each $u\in\P_n^k$ there is a unique permutation $\sigma$ of multi-indices
$W_n^k$ such that
\begin{equation}u = \sum_{\alpha \in W_n^k} S_{\sigma(\alpha)} S_\alpha^*.\label{permutunitary}\end{equation}

As shown by Cuntz in \cite{Cun2}, there exists the following bijective correspondence
between unitaries in $\O_n$ and unital $*$-endomorphisms of $\O_n$ (whose collection we denote
by $\en(\O_n)$). A unitary $u$  in $\O_n$ determines an endomorphism $\lambda_u$ by
$$ \lambda_u(S_i) = u S_i, \;\;\; i=1,\ldots, n. $$
Conversely, if $\rho :\O_n\rightarrow \O_n$ is an endomorphism, then
$\sum_{i=1}^n\rho(S_i)S_i^*=u$ gives a unitary $u\in\O_n$
such that $\rho=\lambda_u$. If the unitary $u$ arises from a permutation $\sigma$ via the formula
\eqref{permutunitary}, the corresponding endomorphism will be sometimes denoted by $\lambda_{\sigma}$.
Composition of endomorphisms corresponds to a `convolution'
multiplication of unitaries as follows:
\begin{equation}\label{convolution}
\lambda_u \circ \lambda_w = \lambda_{\lambda_u(w)u}
\end{equation}
We denote by $\varphi$ the canonical shift:
$$ \varphi(x)=\sum_iS_ixS_i^*, \;\;\; x\in\O_n. $$
If we take $u=\sum_{i, j}S_iS_jS_i^*S_j^*$ then $\varphi=\lambda_u$.
It is well-known that $\varphi$ leaves invariant both $\F_n$ and $\D_n$, and that $\varphi$
commutes with the gauge action $\gamma$.

If $u\in\U(\O_n)$ then for each positive integer $k$ we denote
\begin{equation}\label{uk}
u_k = u \varphi(u) \cdots \varphi^{k-1}(u).
\end{equation}
We agree that $u_k^*$ stands for $(u_k)^*$. If
$\alpha$ and $\beta$ are multi-indices of length $k$ and $m$, respectively, then
$\lambda_u(S_\alpha S_\beta^*)=u_kS_\alpha S_\beta^*u_m^*$. This is established through
a repeated application of the identity $S_i a = \varphi(a)S_i$, valid for all
$i=1,\ldots,n$ and $a \in \O_n$.

Let $z\in\U(\F_n^1)$. Then $\lambda_z$ is an automorphism of $\O_n$ (with inverse $\lambda_{z^*}$), called a Bogolyubov
automorphism. Since $z\in\F_n$, $\lambda_z$ restricts to an automorphism of $\F_n$. Each Bogolyubov automorphism acts as a
unitary transformation on the Hilbert space $\span\{S_i:i=1,\ldots,n\}$ generating $\O_n$.

For algebras $A\subseteq B$ we denote by $\N_B(A)=\{u\in\U(B):uAu^*=A\}$ the normalizer
of $A$ in $B$ and by $A' \cap B=\{a \in A: \forall_{b \in B} \; ab=ba\}$ the relative commutant of $A$ in $B$.

\section{Standard invariant MASAs}\label{masa}

Let $z\in\U(\F_n^1)$. Then $\lambda_z$ is a Bogolyubov automorphism of $\O_n$ and $\A:=\lambda_z(\D_n)$ is a MASA in $\O_n$ (and
in $\F_n$). We will refer to MASAs of this form as standard. Every standard MASA is isomorphic to the $C^*$-algebra of
continuous, complex-valued functions on the full shift on $n$ letters, denoted further by $\mathfrak{C}$ (as it is homeomorphic
to a Cantor set).

\begin{lemma}\label{a-masa}
Let $z\in\U(\F_n^1)$ and denote $\A_1:=\lambda_z(\D_n^1)$. Then for each
positive integer $k$ we have
$$ \lambda_z(\D_n^k)=\A_1\varphi(\A_1)\cdots\varphi^{k-1}(\A_1). $$
Thus $\A$ is the increasing limit of algebras $\A_1\varphi(\A_1)\cdots\varphi^{k-1}(\A_1)$.
\end{lemma}
\begin{proof} At first we verify by induction on $k$ that $\lambda_z\varphi^k(\D_n^1)
=\varphi^k(\A_1)$. Indeed, for $k=1$ we have
$$ \lambda_z\varphi(\D_n^1)=z_2\varphi(\D_n^1)z_2^*=z\varphi(\A_1)z^*=\varphi(\A_1), $$
since $z$ commutes with $\varphi(\A_1)$. Now suppose $\lambda_z\varphi^k(\D_n^1)
=\varphi^k(\A_1)$. Then
$$ \lambda_z\varphi^{k+1}(\D_n^1)=z_{k+2}\varphi^{k+1}(\D_n^1)z_{k+2}^*=
   z\varphi(z_{k+1}\varphi^k(\D_n^1)z_{k+1}^*)z^*=\varphi^{k+1}(\A_1), $$
since $z$ commutes with the range of $\varphi$.

Now we prove the lemma, again proceeding by induction on $k$. Case $k=1$ is just our
definition of $\A_1$. Suppose $\lambda_z(\D_n^k)=\A_1\varphi(\A_1)\cdots\varphi^{k-1}(\A_1)$.
Then we have
$$ \lambda_z(\D_n^{k+1})=\lambda_z(\D_n^k\varphi^k(\D_n^1))=\lambda_z(\D_n^k)
   \lambda_z\varphi^k(\D_n^1) $$
$$ = \A_1\varphi(\A_1)\cdots\varphi^{k-1}(\A_1)\varphi^k(\A_1) $$
by the preceding observation.
\end{proof}

It follows from Lemma \ref{a-masa} that if $w,z\in\U(\F_n^1)$ then
$$ \lambda_w(\D_n)=\lambda_z(\D_n) \;\; \Leftrightarrow \;\; w^*z\in\N_{\F_n^1}(\D_n^1). $$
In particular, $\lambda_z(\D_n)=\D_n$ if and only if $z\in\N_{\F_n^1}(\D_n^1)=\P_n^1\U(\D_n^1)$.

\begin{proposition}\label{normal-masa}
Let $u\in\U(\O_n)$, $z\in\U(\F_n^1)$, and set $\A=\lambda_z(\D_n)$. Then the following hold.
\begin{description}
\item[{\rm (i)}] $\lambda_u(\A)\subseteq\A$ if and only if $\lambda_{\lambda_{z^*}(u)}
(\D_n)\subseteq\D_n$. In particular, if $\lambda_{z^*}(u)\in\N_{\O_n}(\D_n)$ then
$\lambda_u(\A)\subseteq\A$.
\item[{\rm (ii)}] ${\rm Ad}u(\A)\subseteq\A$ if and only if $\lambda_{z^*}(u)\in\N_{\O_n}(\D_n)$.
\end{description}
\end{proposition}
\begin{proof}
(i) Inclusion $\lambda_u(\A)\subseteq\A$ is equivalent to
$\lambda_z^{-1}\lambda_u\lambda_z(\D_n)\subseteq \D_n$. But for $i=1,\ldots,n$ we have
$$ \lambda_z^{-1}\lambda_u\lambda_z(S_i)=\lambda_{z^*}\lambda_u(zS_i)=
\lambda_{z^*}\lambda_u(z)\lambda_u(S_i) = \lambda_{z^*}(uzu^*\cdot uS_i)=\lambda_{z^*}(uzS_i) $$
$$=\lambda_{z^*}(u)\lambda_{z^*}(zS_i)
= \lambda_{z^*}(u)\lambda_{z^*}\lambda_z(S_i)
= \lambda_{z^*}(u)S_i, $$
and thus $\lambda_z^{-1}\lambda_u\lambda_z=\lambda_{\lambda_{z^*}(u)}$. As shown in \cite{Cun2},
if $\lambda_{z^*}(u)$ belongs to the normalizer of $\D_n$ then the endomorphism
$\lambda_{\lambda_{z^*}(u)}$ globally preserves $\D_n$.

(ii) We have $\lambda_z^{-1}{\rm Ad}u\lambda_z={\rm Ad}\lambda_{z^*}(u)$.
\end{proof}

Consider a simple case of $u\in\U(\F_n^1)$. Then $\lambda_{z^*}(u)=z^*uz$ normalizes $\D_n$ if and only if $z^*uz$ is in
$\P_n^1\U(\D_n^1)$. That is, in order to find MASAs of the form $\lambda_z(\D_n)$ which are globally invariant for $\lambda_u$ we
would need to solve the generalized diagonalization problem for $u$: find unitaries $z\in\U(\F_n^1)$ such that
$$ z^*uz=sd, $$
where $s\in\P_n^1$ is a permutation and $d\in\U(\D_n^1)$ is diagonal. More generally, if
$u\in\U(\F_n^k)$, then we would be looking to solve the equation
$$ z_k^*uz_k=sd, $$
where $s\in\P_n^k$, $d\in\U(\D_n^k)$, and $z_k$ is a unitary defined for $z$ in formula (\ref{uk}).

\smallskip
The following example is a rephrasing in
this language of the construction on pages 127--128 in \cite{SZ}.

\begin{example}\label{skalski}
\rm Consider
$$ u=S_1(S_1S_2^*+S_2S_1^*)S_1^*+P_2, $$
a permutation unitary in $\U(\F_2^2)$.
In particular, the diagonal $\D_2$ is invariant under $\lambda_u$. Take
$$ z=(1/\sqrt{2})(1-S_1S_2^*+S_2S_1^*), $$
a unitary in $\U(\F_2^1)$. Then one checks that
$$ \lambda_{z^*}(u)=z_2^*uz_2=S_1S_1S_1^*S_1^*+S_1S_2S_2^*S_2^*+
S_2S_2S_2^*S_1^*+S_2S_1S_1^*S_2^* $$
is a permutation in $\P_2^2$, and thus $\A=\lambda_z(\D_2)$ is an invariant MASA for $\lambda_u$.
This MASA is different from the diagonal $\D_2$, since $z$ does not normalize $\D_2^1$. The endomorphism $\lambda_u$ does not preserve all standard masas: one can check that if $a \in [0,1], b = \sqrt{1-a^2}$, so that $z_a=aS_1 S_1^* - b S_1S_2^* + bS_2S_1^* + a S_2S_2^* $ is a unitary in $\F_2^1$, then $\lambda_u$ preserves $\lambda_{z_a}(\D_2)$ if and only if $a\in \{0,1,\frac{1}{2}\sqrt{2}\}$.
\end{example}

\begin{remark}\label{centralizer}
\rm Let $u\in\N_{\F_n^k}(\D_n^k)$ and denote $C(u)=\{w\in\U(\F_n^k):wu=uw\}$, the centralizer of
$u$ in $\U(\F_n^k)$. To find invariant MASAs for $\lambda_u$
one could look for unitaries $z\in\U(\F_n^1)$ such that $z_k\in C(u)\N_{\F_n^k}(\D_n^k)$.
If such a $z$ does not normalize $\D_n^1$ then $\A=\lambda_z(\D_n)$ is different from $\D_n$.
\end{remark}

\begin{example}\label{thompson}
\rm Consider the unitary
$$ u = S_1S_1S_1^*+S_1S_2S_1^*S_2^*+S_2S_2^*S_2^* $$
in $\SS_2$. The unitary $u$ is in the normalizer of the
diagonal, i.e.\ ${\rm Ad}u(\D_2)=\D_2$. We ask if there exist other standard MASAs in $\O_2$ invariant under ${\rm Ad}u$. So let
$\theta\in[0,2\pi)$, $a,b\in{\mathbb C}$ with $|a|^2+|b|^2=1$, and consider the unitary
$$ z=e^{i\theta}aS_1S_1^*+e^{i\theta}bS_2S_1^*-\overline{b}S_1S_2^*+\overline{a}S_2S_2^* $$
in $\U(\F_2^1)$. According to Proposition
\ref{normal-masa}, we must determine for what $z$ the unitary $\lambda_{z^*}(u)$ belongs to the normalizer of $\D_2$. In
particular, for each projection $p\in\D_2$ we must have $\lambda_{z^*}(u)p \lambda_{z^*}(u)^*\in\D_2$. Decomposing
$\lambda_{z^*}(u)p\lambda_{z^*}(u)^*$ as a sum of terms belonging to spectral subspaces $\O_2^{(k)}$, $k\in\Z$, we get that the
term in $\O_2^{(2)}$ is $\lambda_{z^*}(S_1S_1S_1^*)p\lambda_{z^*}(S_2S_2S_2^*)$. Writing this expression explicitly for
$p=S_1S_1S_1^*S_1^*$ (for example) we see that it is equal to $0$ if and only if either $a=0$ or $b=0$. But then $z$ normalizes
$\D_2$ and $\lambda_z(\D_2)=\D_2$. Therefore the only standard MASA of $\O_2$ invariant under ${\rm Ad}u$ is the diagonal $\D_2$.
\end{example}

As it turns out, there exist unitary elements of $\F_n^k$ whose associated endomorphisms do not leave invariant any standard MASA
of $\O_n$. In Section \ref{izumi}, below, we show that this is the case for a certain endomorphism of $\O_2$ constructed by Izumi
(\cite{I}) in connection with Longo's sector theory (\cite{L2}) and Watatani's index theory for $C^*$-algebras (\cite{W}).

\section{Endomorphisms preserving the diagonal}\label{enddiag}

Condition (i) of Proposition \ref{normal-masa} motivates the following question: for what
unitaries $w\in\U(\O_n)$ we have $\lambda_w(\D_n)\subseteq\D_n$? Certainly this is the case
for those unitaries which normalize $\D_n$. An interesting question is what other unitaries,
not belonging to ${\mathcal N}_{\O_n}(\D_n)$, give rise to endomorphisms which preserve
the diagonal? Clearly, $\lambda_w(\D_n)\subseteq\D_n$ if and only if $\lambda_w(\D_n^1)
\subseteq\D_n$ and $\lambda_w\varphi^r(\D_n^1)\subseteq\D_n$ for all $r=1,2,\ldots$.
Unfortunately, these conditions are not easy to check in general. However, if $w$ is a unitary
in the algebraic part of $\F_n$ then they reduce to a finite calculation, as explained
in the following theorem.

\begin{theorem}\label{finitecond}
Let $k \in \Naturali$ and let $w\in\U(\F_n^k)$. For $i,j=1,\ldots,n$ let $E_{ij}:\F_n^k \rightarrow \F_n^{k-1}$ be linear maps
determined by the condition that $a=\sum_{i,j=1}^n E_{ij}(a)\varphi^{k-1}(S_iS_j^*)$ for all $a\in\F_n^k$. Define by induction an
increasing sequence of unital selfadjoint subspaces $\mathfrak{S}_r$ of $\F_n^{k-1}$ so that
\[ \begin{aligned}
\mathfrak{S}_1 & = \span\{ E_{jj}(wxw^*): x\in\D_n^1,\; j=1,\ldots,n \}, \\
\widetilde{\mathfrak{S}}_{r+1} & = \span\{ E_{jj}((\Ad w\circ\varphi)(x)):
x\in\mathfrak{S}_r,\; j=1,\ldots,n\}, \\
\mathfrak{S}_{r+1} & = \mathfrak{S}_r + \widetilde{\mathfrak{S}}_{r+1}.
\end{aligned} \]
We agree that $\mathfrak{S}_0={\mathbb C}1$.
Let $R$ be the smallest integer such that $\mathfrak{S}_R=\mathfrak{S}_{R-1}$. Then $\lambda_w(\D_n)
\subseteq\D_n$ if and only if $\lambda_w(\D_n^R)\subseteq\D_n$.
\end{theorem}
\begin{proof}
Let $\lambda_w(\D_n^R)\subseteq\D_n$. We show by induction on $r=1,2,\ldots$ that
\begin{description}
\item[\hspace{2mm}{\rm (i)}] $\mathfrak{S}_r \subseteq \D_n^{k-1}$;
\item[\hspace{1mm}{\rm (ii)}] $(\Ad w \circ \varphi)(\mathfrak{S}_{r-1})\subseteq
\mathfrak{S}_r\varphi^{k-1}(\D_n^1)$;
\item[{\rm (iii)}] $\lambda_w\varphi^{r-1}(\D_n^1) \subseteq\D_n$;
\item[{\rm (iv)}] for each $x\in\D_n^1$ and $r$ there exist
elements $x_{j_1\ldots j_r}$ of $\mathfrak{S}_r$ such that
\begin{equation}\label{x}
\lambda_w\varphi^{r-1}(x) = \sum_{j_1,\ldots,j_r=1}^n x_{j_1\ldots j_r}
\varphi^{k-1}(P_{j_r})\varphi^k(P_{j_{r-1}})\cdots\varphi^{k+r-2}(P_{j_1});
\end{equation}
\item[\hspace{1mm}{\rm (v)}] $\mathfrak{S}_r$ is spanned by $\mathfrak{S}_{r-1}$ and
elements $x_{j_1\ldots j_r}$ appearing in formula (\ref{x}) for all $x\in\D_n^1$.
\end{description}
Base step $r=1$. By hypothesis, $\lambda_w(\D_n^1)=w\D_n^1w^*$ is contained in $\D_n^k$.
Thus for $x\in\D_n^1$ we have
$$ \lambda_w(x)=\sum_{j=1}^n E_{jj}(wxw^*)\varphi^{k-1}(P_j) $$
and all $E_{jj}(wxw^*)$ must belong to $\D_n^{k-1}$. Whence $\mathfrak{S}_1\subseteq\D_n^{k-1}$ and formula (\ref{x}) holds for
$r=1$. Condition (v) follows.

For the inductive step, suppose the claim holds for all $m=1,\ldots,r$. Let $x\in\D_n^1$. Then
by the inductive hypothesis, and taking into account that $w$ commutes with the range of
$\varphi^l$ if $l\geq k$, we get
\[ \begin{aligned}
\lambda_w\varphi^r(x) & = (\Ad w \circ \varphi)\lambda_w\varphi^{r-1}(x) \\
 & = \sum_{j_1,\ldots,j_r=1}^n (\Ad w \circ \varphi)(x_{j_1\ldots j_r})
\varphi^k(P_{j_r})\cdots\varphi^{k+r-1}(P_{j_1})
\end{aligned} \]
for some $x_{j_1\ldots j_r}\in\mathfrak{S}_r$. Firstly, suppose that $r\leq R-1$. Then
$\varphi^r(x)\in\D_n^R$ and $\lambda_w\varphi^r(x)$ belongs to $\D_n$, by assumption. Thus
each $(\Ad w \circ \varphi)(x_{j_1\ldots j_r})$ is in $\D_n^k$ and we have
$$ (\Ad w \circ \varphi)(x_{j_1\ldots j_r}) = \sum_{j=1}^n E_{jj}
   ((\Ad w \circ \varphi)(x_{j_1\ldots j_r}))\varphi^{k-1}(P_j). $$
Setting
$$ x_{j_1\ldots j_r j_{r+1}}=E_{j_{r+1}j_{r+1}}((\Ad w \circ \varphi)(x_{j_1\ldots j_r})) $$
we see that formula (\ref{x}) holds for $r+1$, and also that
$\mathfrak{S}_{r+1}\subseteq\D_n^{k-1}$ and
$(\Ad w \circ \varphi)(\mathfrak{S}_r)\subseteq\mathfrak{S}_{r+1}\varphi^{k-1}(\D_n^1)$.
This establishes the inductive step in the present case.

Secondly, suppose that $r\geq R$. Then
all elements $x_{j_1\ldots j_r}$ belong to $\mathfrak{S}_r=\mathfrak{S}_{r-1}$. Thus,
by the inductive hypothesis, $(\Ad w \circ \varphi)(x_{j_1\ldots j_r})$ is in
$\mathfrak{S}_r\varphi^{k-1}(\D_n^1)\subseteq\D_n^{k-1}\varphi^{k-1}(\D_n^1)$.
Consequently, $\lambda_w\varphi^r(x)$ belongs to $\D_n$, and the rest of the argument is
similar to the preceding case.

Validity of condition (iii) for all $r \in \Naturali$ implies that $\lambda_w(\D_n) \subset \D_n$ and the proof is finished.
\end{proof}

The simplest possible case to consider from the point of view of Theorem \ref{finitecond}
is when already $\mathfrak{S}_1=\mathfrak{S}_0$. This leads to the following.

\begin{corollary}\label{easydninvariant}
Let $w$ be a unitary element of $\F_n^k$. If $w\D_n^1w^* =\varphi^{k-1}(\D_n^1)$
then $\lambda_w(\D_n)\subseteq\D_n$.
\end{corollary}
\begin{proof}
Indeed, if $w\D_n^1w^* =\varphi^{k-1}(\D_n^1)$ then $\mathfrak{S}_1=\mathfrak{S}_0$
and $\lambda_w(\D_n^1)\subseteq\D_n$.
\end{proof}

Corollary \ref{easydninvariant} and Proposition \ref{normal-masa} lead to the following explicit condition guaranteeing that the
endomorphism associated to a unitary in the algebraic part of $\F_n$ leaves a given standard MASA invariant.

\begin{corollary}\label{easyminvariant}
Let $u\in\F_n^k$ and $z\in\U(\F_n^1)$. If $u(z\D_n^1z^*)u^* =\varphi^{k-1}(z\D_n^1z^*)$
then $\A=\lambda_z(\D_n^1)$ is $\lambda_u$-invariant.
\end{corollary}

\begin{example}\label{nonorm}
\rm One notices that Corollary \ref{easydninvariant} easily gives examples of unitaries
$w$ such that $\lambda_w(\D_n)\subseteq\D_n$ but $w\not\in\N_{\O_n}(\D_n)$. For
instance, if $a,b,c,d\in{\mathbb C}$ are such that $|a|^2+|b|^2=|c|^2+|d|^2=1$ then
\begin{eqnarray*}
w & = & S_1(aS_1S_1^*+bS_1S_2^*)S_1^* + S_1(cS_2S_1^*+dS_2S_2^*)S_2^* \\
  & + & S_2(-\overline{b}S_1S_1^*+\overline{a}S_1S_2^*)S_1^*
    + S_2(-\overline{d}S_2S_1^*+\overline{c}S_2S_2^*)S_2^*
\end{eqnarray*}
satisfies the condition of Corollary \ref{easydninvariant} (with $n=k=2$) and thus
$\lambda_w(\D_2)\subseteq\D_2$, but $w$ belongs to $\N_{\O_2}(\D_2)$ only if $ab=cd=0$.
\end{example}

Motivated by \cite{SZ} (where a similar construction was used when an endomorphism different from the canonical shift was shown
to reduce to the classical shift on a non-diagonal invariant standard MASA) we also consider the following sufficient condition which
guarantees that $\lambda_w(\D_n)\subseteq\D_n$:

\vspace{2mm}\noindent
[D1] $\lambda_w(\D_n^1)\subseteq\D_n$,

\vspace{2mm}\noindent
[D2] $\lambda_w\varphi^r(x)=\varphi^r\lambda_w(x)$ for all $x\in\D_n^1$ and $r=1,2,\ldots$

\vspace{2mm} A simple inductive argument shows that

\vspace{2mm}\noindent
[D2] $\Leftrightarrow$ $\displaystyle{w\in\bigcap_{r=1}^\infty
\left(\varphi^r\lambda_w(\D_n^1)\right)'\cap\O_n}$.

\vspace{2mm}\noindent
Thus, since $\F_n^k$ and $\varphi^k(\O_n)$
commute, for $w\in\U(\F_n^k)$ conditions [D1] and [D2] are equivalent to the following:

\vspace{2mm}\noindent
[DF1] $w\D_n^1w^*\subseteq\D_n^k$,

\vspace{2mm}\noindent
[DF2] $\displaystyle{w\in\bigcap_{r=1}^{k-1}\left(\varphi^r(w\D_n^1w^*)\right)'\cap\F_n^k}$.

\vspace{2mm}\noindent
In this way we obtain the following.

\begin{proposition}\label{dninvariant}
Let $w$ be a unitary element of $\F_n^k$. Denote by $B_w$ the $C^*$-algebra
generated by $\{\varphi^r(wxw^*): x\in\D_n^1,\,r=1,\ldots,k-1\}$. If
$$ w\D_n^1w^* \subseteq \D_n^k \;\;\; \text{and} \;\;\; w \in B_w'\cap\F_n^k $$
then $\lambda_w(\D_n)\subseteq\D_n$.
\end{proposition}

A very special case of the situation described in the preceding proposition takes
place when $w\in\U(\F_n^k)$ and $w\D_n^1w^* =\varphi^{k-1}(\D_n^1)$, the case considered
in Corollary \ref{easydninvariant}.

\section{Endomorphisms commuting with\\  Bogolyubov automorphisms}\label{bogolyubov}

The methods of computing entropy of permutation endomorphisms of $O_n$ used in \cite{SZ} were based on analysis of the behaviour
of the transformations in question on invariant standard MASAs. It was noted there that even in simple cases entropy properties
of $\rho|_{\mathcal{A}_1}$ and $\rho|_{\mathcal{A}_2}$ for two invariant standard MASAs $\mathcal{A}_1$ and $\mathcal{A}_2$ can
be very different. In this section, we discuss endomorphisms of $\O_n$ whose restrictions to all standard MASAs are identical.

We begin with a simple and well-known observation:

\begin{lemma}\label{commasa}
Let $u\in\U(\O_n)$. If the endomorphism $\lambda_u$ commutes with the gauge action, then $u \in \F_n$.
\end{lemma}
\begin{proof}
The assumed commutation relation implies that for each $i=1,\ldots,n$ and $t \in U(1)$ we have
\[ \gamma_t(u) \gamma_t(S_i) = u \gamma_t(S_i).\]
This implies that $\gamma_t(u) = u$ for all $t\in U(1)$ and thus $u\in\F_n$.
\end{proof}

By virtue of the above lemma, endomorphisms commuting with the gauge action leave the core UHF subalgebra
$\F_n$ invariant. Interesting examples of endomorphisms of $\O_n$ that leave $\F_n$ invariant and yet
their associated unitaries are not in $\F_n$ were recently found in \cite{CRS}.

We denote by $\sg(k)$ the group of all permutations of $k$ letters.

\begin{definition}\label{induced}
\rm Let $\sigma$ be a permutation of $W_n^k$. We call $\sigma$ {\em induced} if there exists a permutation
$\omega \in \sg(k)$ such that for all $\alpha \in W_n^k$ we have
\[ \sigma (\alpha) = (\alpha_{\omega(1)}, \ldots, \alpha_{\omega(k)}).\]
\end{definition}

The following fact is a consequence of the classical theorem of Hermann Weyl.

\begin{lemma}\label{Weyl}
Let $k \in {\mathbb N}$ and let $u$ be a unitary in $\F_n^k$. The endomorphism $\lambda_u$ commutes
with all Bogolyubov automorphisms of $\O_n$ if and only if $u$ is a linear combination of permutation unitaries
associated with induced permutations of $W_n^k$.
\end{lemma}
\begin{proof}
Let $u$ be as in the statement of the lemma and let $z\in\U(\F_n^1)$, so that $\lambda_z$ is a Bogolyubov
automorphism of $\O_n$. By (\ref{convolution}), endomorphisms $\lambda_u$ and $\lambda_z$
commute if and only if
\[ \lambda_u(z)u = \lambda_z(u)z.  \]
Since $\lambda_u(z)=uzu^*$ and $\lambda_z(u)=z_kuz_k^*$ (with $z_k$ as in (\ref{uk})), this is equivalent to
\begin{equation}\label{uzk}
 u= z_k u z_k^*.
\end{equation}
Under the identification $\F_n^k \cong M_n(\cc)^{\otimes k}$, unitary $z_k$ corresponds to the tensor
product $z^{\otimes k}$. Then equation (\ref{uzk}) takes the form
\begin{equation} \label{comBog}
 u = z^{\otimes k} u z^{*\otimes k}.
\end{equation}
View now this equation as an equation for a fixed matrix $u \in M_n(\cc)^{\otimes k}$ and
arbitrary unitary $z\in M_n(\cc)$. It means that
$$ u\in \{ z^{\otimes k}:z\in U(n)\}'=\{ z^{\otimes k}:z\in SU(n)\}', $$
so that $u$ is in the commutant of
the $k$-th tensor power of the standard representation of $SU(n)$. A classical result of Weyl says that this
commutant is equal to the image of the natural representation of the permutation group
$\sg(k)$ in $M_n(\cc)^{\otimes k}$ (see  \S 9.1.1 in \cite{H}). This means that $u$ commutes with
all the Bogolyubov automorphisms if and only if it is a linear combination of unitary matrices representating the
permutations in $\sg(k)$. It remains to observe that under the identification of $\F_n^k$ with
$M_n(\cc)^{\otimes k}$ the latter matrices correspond to the unitaries associated
with induced permutations.
\end{proof}

Lemmas \ref{commasa} and \ref{Weyl} almost characterize unitaries that lead to endomorphisms commuting with all Bogolyubov
automorphisms;  we can also have $u \in \F_n \setminus \bigcup_{k\in {\mathbb N}} \F_n^k$, but this can happen only in a trivial
way (i.e.\ $u$ will be a norm limit of unitaries of the type described in Lemma \ref{Weyl}). This has been observed (in the
context of the UHF algebras) by Price in \cite{Price}. We formulate it in the next theorem; for the completeness we also present
the proof. Let
\begin{equation}\label{fixed}
{\bf F} = \bigcap_{z \in \U(\F_n^1)} \{ x\in\O_n: \lambda_z(x)=x \}.
\end{equation}

\begin{theorem}
Let $u \in \O_n$ be unitary. The endomorphism $\lambda_u$ commutes with all Bogolyubov
automorphisms of $\O_n$ if and only if $u$ is in the closed linear span of unitaries associated
with induced permutations.
\end{theorem}
\begin{proof}
By Lemmas \ref{commasa} and \ref{Weyl} together with their proofs it suffices to show that $\Fix \cap \bigcup_{k \in \Naturali}
\F_n^k$ is dense in $\Fix$. Let $\tau \in \F_n^*$ denote the faithful trace of the UHF algebra $\F_n$. Let $k \in \Naturali$.
With respect to the natural isomorphism $\iota_k: \F_n \to \F_n \otimes {M_n}^{\otimes k}$ we have $\tau = (\tau \otimes
\tr_{n^k})\circ \iota_k$, where $\tr_{n^k}$ is the normalised trace on ${M_{n}}^{\otimes k}$. The identification of
${M_{n}}^{\otimes k}$ with $\F_n^k$ in this picture corresponds to the mapping $\theta_k: {M_{n}}^{\otimes k} \to \F_n$ defined
by $\theta_k(x) = {\iota_k}^{-1}( 1_{\F_n} \otimes x)$, $ x \in {M_{n}}^{\otimes k}$. Consider the conditional expectation
$E_k:\F_n \to \F_n^k$ given by
\[ E_k = \theta_k \circ(\tau \otimes \id_{{M_{n}}^{\otimes k}}) \circ \iota_k.\]
As $\bigcup_{k \in \Naturali} \F_n^k$ is dense in $\F_n$,
\[
E_k(x) \stackrel{k \to \infty}{\longrightarrow} x, \;\;\;x \in \F_n.\] Let $\lambda_z$ be a Bogolyubov automorphism ($z \in \F^1_n$).
Then $E_k \circ \lambda_z = \lambda_z \circ E_k$ -- the commutation relation is easily checked on elements in $\F_n^l$ ($l\geq k$) and by
contractivity of the maps involved has to hold on the whole $\F_n$. This implies that for each $k \in \Naturali$ there is $E_k
(\Fix) = \Fix$, which together with the displayed formula yields the statement formulated in the first sentence of the proof.
\end{proof}

The above theorem, or rather its proof, shows that we have (recall the notation in \eqref{fixed})
\[\Fix \cap \O_n =  \Fix  \cap \F_n  \approx
C^*(\sg(\infty))\]
(with $\sg(\infty)$ denoting the group of finite permutations of ${\mathbb N}$), as noted by Bratteli and Evans in \cite{BE}.

\begin{corollary}\label{indcom}
Let $\sigma$ be a permutation of $W_n^k$. Then the corresponding endomorphism $\lambda_{\sigma}$
commutes with all Bogolyubov automorphisms of $\O_n$ if and only if $\sigma$ is induced.
\end{corollary}
\begin{proof}
Follows from Lemma \ref{Weyl} and a simple observation that a permutation unitary in $\F_n^k$ can be a linear combination of
permutation unitaries in $\F_n^k$ only in a trivial way.
\end{proof}

Thus the class of induced permutation endomorphisms is the class of endomorphisms that not only leave each standard MASA
invariant, but also induce the same homeomorphisms of the underlying Cantor set in each case. When $k=2$, induced permutation
endomorphisms are the identity morphism and the canonical shift, but already when $k=3$ we get some new maps, for example the
endomorphism associated to the unitary $\sum_{i,j,l=1}^n S_l S_j S_i S_j^* S_i^* S_l^*$. Note that the product of induced
permutation endomorphisms is an induced permutation endomorphism, with the action on the level of permutations in
$\bigcup_{k\in{\mathbb N}}\sg(k)$ given by the convolution multiplication described by formula $(7)$ of \cite{S} (note that in
\cite{S} the convolution multiplication is defined for permutations of multi-indices). Each such endomorphism on the diagonal (or
any other standard MASA) reduces to a transformation which has a simple combinatorial description as a substitution map.

\begin{commento}\label{bogocomment}
\rm  In view of the discussion in \cite{Sk} and the simple form of the induced permutation endomorphisms it is natural to expect
that  for such endomorphisms their Voiculescu topological entropy can be detected already from the restriction to the canonical
diagonal MASA. We also conjecture  that for an endomorphism $\lambda$ associated to an induced permutation the entropies of $\lambda$
and $\lambda \circ \alpha$ coincide  for all Bogolyubov automorphisms $\alpha$. Note that in that situation we have $(\lambda \circ
\alpha)^m=\lambda^m \circ \alpha^m$. A combination of results in Section \ref{izumi}, below, and a straightforward symbolic dynamics
computation allowed to show in \cite{Sk2} that in general for an endomorphism $\lambda$ of $\O_n$ (even induced by a unitary in
$\F_n$) and a Bogolyubov automorphism $\alpha$ the entropies of $\lambda$ and $\lambda \circ \alpha$ may be different  (although, as
follows from \cite{DS}, the entropy of $\alpha$ is $0$).
\end{commento}

Finally, we observe that there exist endomorphisms preserving all standard MASAs which nevertheless do
not commute with all Bogolyubov automorphisms, as the following example demonstrates.

\begin{example}\label{nobogo}
\rm Let $k\geq 2$ and let $\sigma\in\sg(k)$ be such that $\sigma(1)=k$ and $\sigma(k)=1$. Let
$u_1$ be the unitary in $\F_n^k$ induced by $\sigma$, and let $u_2$ be an arbitray unitary
in $\varphi(\F_n^{k-1})$. Set $u=u_1 u_2$. Then for any $z\in\U(\F_n^1)$ we have
$\lambda_z(u_1)=u_1$ and $\lambda_z(u_2)\in\varphi(\F_n^{k-1})$. Note that for each $i,j=1,\ldots,n$
there is $u_1 S_i S_j^* = \varphi^{k-1} (S_i S_j^*) u_1$. Consequently,
$$ \lambda_z(u)\D_n^1\lambda_z(u^*)=u_1\lambda_z(u_2)\D_n^1\lambda_z(u_2^*)u_1^*=
u_1\D_n^1 u_1^* = \varphi^{k-1}(\D_n^1). $$
Thus $\lambda_u$ preserves all standard MASAs by Corollary \ref{easyminvariant}. However,
using freedom of choice of $u_2$, one can arrange that $u$ does not satisfy conditions of
Lemma \ref{Weyl} (or Corollary \ref{indcom}), and thus it does not commute with all
Bogolyubov automorphisms. A particular example of such situation is given by ($n=2)$
$u_1 = S_1S_1 S_1^* S_1^* + S_1 S_2 S_1^* S_2^* + S_2 S_1 S_2^* S_1^* + S_2 S_2 S_2 S_2^*$,
$u_2 =  S_1S_1 S_1^* S_2^* + S_1 S_2  S_1^* S_1^* + S_2 S_1 S_2 S_2^* + S_2 S_2 S_2^* S_1^*$.
The endomorphism $\lambda_{u_1u_2}$ is then a permutation endomorphism denoted by $\rho_{1342}$ in \cite{SZ}.
We recommend comparing its properties with these of the endomorphism described in Example \ref{skalski}.
\end{example}

\section{Izumi's examples of real sectors}\label{izumi}

In \cite{I}, Izumi studied certain explicit examples of endomorphisms
of Cuntz algebras motivated by subfactor theory.
His examples are square roots of canonical endomorphisms in the sense of Longo (\cite{L2}) with finite Watatani
indices (\cite{W}). Izumi's methods are based on the following proposition, which is a special case
of Proposition 2.5 in \cite{I}:

\begin{proposition}[\cite{I}]\label{quoted}
Let $\lambda \in \en(\O_n)$. If there exists an isometry $w \in \O_n$ such that $wx = \lambda^2(x) w$ for all
$x \in \O_n$ and $w^*\lambda(w)=\frac{1}{d}$ for some $d >0$, then $E_{\lambda}:\O_n \to \lambda(\O_n)$ defined by
\[ E_{\lambda}(x) = \lambda(w^* \lambda(x) w),\;\;x\in \O_n,\]
is a conditional expectation with Watatani index $d^2$.
\end{proposition}

Endomorphism $\lambda$ as above is called a \emph{real sector}, as it is its own
conjugate endomorphism (\cite{L2}). Examples of
real sectors satisfying the conditions in the above proposition
cannot be obtained for $w=S_i$ and $\lambda$ being a permutation
endomorphism. Nevertheless, the endomorphism constructed in Example 3.7 of \cite{I} is of the form $\lambda=\lambda_{\sigma} \circ
\beta$, where $\beta$ is a Bogolyubov automorphism; moreover $\lambda^2$ is also a permutation endomorphism. Below we give proofs of these facts.

Let $G$ be a finite abelian group of cardinality $n$, written additively. Let
$\la\cdot, \cdot \ra:G \times G \to \Toro$ be a symmetric duality
bracket  satisfying the usual conditions ($g,g',h \in G$)
\[ \overline{\la g, h\ra} = \la -g, h\ra, \;\; \sum_{h \in G} \la h, g\ra =
 \begin{cases}0 & \text{if }  g\neq e \\ N & \text{if } g = e\end{cases} ,\]
\[ \la g,h \ra \la g',h\ra = \la g+g',h \ra, \;\; \la g, h \ra = \la h, g\ra. \]
If $G=\Z_n$ is a cyclic group then one can
put $\la k, l\ra:=\exp(\frac{2\pi i(kl)}{n})$. In this section, we
will use elements of $G$ as indices of generating isometries in $\O_n$. For $g\in G$, define unitaries
$U(g) \in \F_n \subset \O_n$ as
\[U(g) = \sum_{h \in G} \la g,h\ra S_h S_h^*. \]
Note that they give rise to a `unitary representation' of $G$. Now define
an endomorphism $\lambda\in \en(\O_n)$ by
\[ \lambda (S_g) = \frac{1}{\sqrt{n}} \sum_{h \in G} \la g, h \ra S_h U(g)^*.\]
The endomorphism $\lambda$ does not leave $\D_n$ invariant (unless $n=1$), as we have
\[ \lambda(S_g S_g^*) = \frac{1}{n} \sum_{h,k \in G} \la g, h-k \ra S_h S_k^*,\]
so for example \[\lambda(S_e S_e^*) = \frac{1}{n} \sum_{h,k \in G}  S_h S_k^*.\]
The unitary associated with $\lambda$ is equal to
\begin{equation} \label{Izunit}
v_\lambda= \frac{1}{\sqrt{n}} \sum_{g,h,l \in G} \la g,h-l \ra S_h S_l S_l^* S_g^*.
\end{equation}

For each $h \in G$ define an isometry
\[\widetilde{S}_h = \frac{1}{\sqrt{n}} \sum_{a \in G} \la h, a\ra S_a\]
and let $\beta \in \aut(\O_n)$ be given by
\[\beta(S_h) = \widetilde{S}_h, \;\;\; h \in G.\]
It is easy to see that $\beta$ is a Bogolyubov automorphism. Note that $\lambda':= \lambda \circ \beta$ is a
permutation endomorphism. Indeed, for each $h \in G$ we have
$$ \begin{aligned} \lambda'(S_h) &= \lambda(\widetilde{S}_h) = \frac{1}{\sqrt{n}}
\sum_{a \in G} \la h, a\ra \lambda(S_a) = \frac{1}{n} \sum_{a,b,c \in G}
\la h, a \ra \la b, a \ra \la -c, a \ra S_b S_c S_c^*
\\ &= \sum_{b \in G} S_b S_{h+b}S_{h+b}^*.
\end{aligned} $$

Let us now compute $\lambda^2$ (recall it is a canonical endomorphism in the sense of Longo).
For $g\in G$ we have
$$ \begin{aligned}
\lambda(U(g))  & = \sum_{h \in G} \la g,h\ra  \frac{1}{n} \sum_{k,l \in G} \la h, k-l \ra S_k S_l^* =
\frac{1}{n}  \sum_{h,k,l \in G} \la h, g+ k-l \ra S_k S_l^*  \\
 & = \sum_{k \in G} S_k {S_{g+k}}^*,
\end{aligned} $$
and  further
$$ \begin{aligned}\lambda^2(S_g) &= \frac{1}{\sqrt{n}} \sum_{h \in G} \la g, h \ra \lambda(S_h) \lambda(U(g))^* \\
 & = \frac{1}{n} \sum_{h \in G} \la g, h \ra  \sum_{l \in G} \la h,
 l \ra S_l U(h)^* \sum_{k \in G} S_{g+k} S_{k}^* \\
&= \frac{1}{n} \sum_{h,l,k \in G} \la g+l, h \ra  S_l  \sum_{a \in G} \la h,- a\ra S_a S_a^*  S_{g+k} S_{k}^* \\
& = \frac{1}{n} \sum_{h,l,k \in G} \la g+l, h \ra  S_l   \la h,- g - k\ra   S_{g+k} S_{k}^*\\
 & = \frac{1}{n} \sum_{h,l,k \in G} \la l-k, h \ra  S_l    S_{g+k} S_{k}^* \\
& = \sum_{k \in G}   S_k      S_{g+k} S_{k}^*.
\end{aligned}  $$
We can see from the formula above that $\lambda^2$ is a permutation endomorphism; its associated
unitary is equal to
\[ v_{\lambda^2} = \sum_{g,h \in G}   S_g      S_{h+g} S_{g}^* S_h^*.  \]

Below, we take a closer look at the simplest nontrivial case, $G=\Z_2$.

\begin{example}\label{izuexa}
\rm Let $G=\Z_2=\{0,1\}$ be equipped with the natural duality bracket; $\la 1, 1 \ra = -1$
and all other brackets take value $1$. The Izumi endomorphism discussed above is then given by
$$ \begin{aligned}
 \lambda(S_0)  & = \frac{1}{\sqrt{2}} (S_0 + S_1), \\
 \lambda(S_1) & = \frac{1}{\sqrt{2}} (S_0 S_0 S_0^* + S_1 S_1 S_1^*  - S_1 S_0 S_0^* - S_0 S_1 S_1^*).
\end{aligned} $$
The unitary associated with $\lambda$ is given by
$$ \begin{aligned} v_\lambda &= \frac{1}{\sqrt{2}} \left(S_{00,00} + S_{11,11} +
S_{01,10} + S_{00,01} + S_{10,00} +S_{11,10} - S_{10,01}  -S_{01,11}\right) \\
 & = \frac{1}{\sqrt{2}} (   S_{00,01} + S_{11,11} - S_{01,11} - S_{10,01} + S_{0,0} + S_{1,0}) \\
 &= \frac{1}{\sqrt{2}} (   1 + S_{0,1} + S_{1,0} - 2 S_{01,11} - 2 S_{10,01} ),
\end{aligned} $$
where we write $S_{i,j}:=S_i S_j^*$, $S_{ij,kl} := S_i S_j S_k^* S_l^*$ for all $i,j,k,l \in\{0,1\}$.

The permutation endomorphism $\lambda'$ considered above is induced by the unitary
$$ v_{\lambda'}= S_{00,00}+ S_{01,11} + S_{11, 01} + S_{10,10} $$
and equals $\rho_{24}$ of \cite{SZ}. Endomorphism $\lambda^2$ is associated to the unitary
$$ v_{\lambda^2} = S_{00,00}+ S_{01,01}+ S_{11, 10} + S_{10,11} $$
and equals $\rho_{243}$ of \cite{SZ}.

It turns out that endomorphism $\lambda$ does not leave any standard MASA of $\O_2$ invariant. Indeed,
let $z \in \U(\F_2^1)$. Then $\lambda$ preserves MASA $\lambda_{z^*}(\D_2)$ if and only if
the endomorphism associated with $X=\lambda_z(v_\lambda)$ preserves $\D_2$ (Proposition
\ref{normal-masa}). For this, it is necessary that $X\D_2^1X^*\subset\D_2^2$. So let
$$ z= aS_{0,0} +bS_{0,1} -\overline{b}S_{1,0} + \overline{a}S_{1,1} $$
for some $a,b\in{\mathbb C}$ with $|a|^2+|b|^2=1$ (clearly, it suffices to consider unitaries
$z$ with determinant $1$). Then we have  $X= \frac{1}{\sqrt{2}}zYz^*$, where
\[ Y = 1 +  S_{0,1} +  S_{1,0} - 2S_0 zS_1S_1^* z^* S_1^* - 2S_1zS_0 S_0^*z^* S_1^*.  \]
We verify that $XP_0 X^*$ does not belong to $\D_2^2$ for any choice of parameters $a,b$.
Indeed, since $z$ is in the commutant of $\varphi(\F_2^1)$, for $XP_0X^*$ to be in $\D_2^2$
it is necessary that $Y(z^*P_0 z)Y^*$ is in $\F_2^1\varphi(\D_2^1)$. For computational convenience,
we write $Y$ and $z^*P_0z$ as 4-by-4 matrices in block form. Then
$$ Y = \left[ \begin{array}{cc} 1 & Y_1 \\ 1 & Y_2 \end{array} \right] \;\;\;\;\; \text{and} \;\;\;\;\;
    z^*P_0z = \left[ \begin{array}{cc} c1 & d1 \\ \overline{d}1 & (1-c)1 \end{array} \right], $$
where $c=|a|^2$, $d=\overline{a}b$ and $Y_1, Y_2$ are self-adjoint
2-by-2 matrices such that
$$ Y_1 = \left[ \begin{array}{cc} 2a\overline{a}-1 & -2ab \\ -2\overline{ab} & 1-2a\overline{a}
    \end{array} \right], \;\;\;\;\; Y_2 = \left[ \begin{array}{cc} 1-2a\overline{a} & 2ab \\
    2\overline{ab} & 2a\overline{a}-1 \end{array} \right]. $$
Thus $Y(z^*P_0 z)Y^*$ equals
$$  \left[ \begin{array}{cc} c1+(d+\overline{d})Y_1+(1-c)Y_1^2 &
    c1+dY_2+\overline{d}Y_1+(1-c)Y_1Y_2 \\ c1+\overline{d}Y_2+dY_1+(1-c)Y_2Y_1 &
    c1+(d+\overline{d})Y_2+(1-c)Y_2^2  \end{array} \right]. $$
This is an element of $\F_2^1\varphi(\D_2^1)$ if and only if all its four blocks are diagonal 2-by-2 matrices.
An elementary calculation shows that this happens if and only if either $a=0$ or $b=0$. But for these
two choices of the parameters one can easily see that $XP_0X^*$ does not belong to $\D_2^2$.
\end{example}

\smallskip\noindent
Jeong Hee Hong \\
Department of Applied Mathematics \\
Korea Maritime University \\
Busan 606--791, South Korea \\
E-mail: hongjh@hhu.ac.kr \\

\smallskip\noindent
Adam Skalski \\
Department of Mathematics and Statistics \\
Lancaster University \\
Lancaster, LA1 4YF, United Kingdom\footnote{Permanent address: Faculty of Mathematics and Computer Science,
University of \L \'{o}d{\'z}, ul. Banacha 22, 90--238  \L \'{o}d{\'z}, Poland} \\
E-mail: a.skalski@lancaster.ac.uk \\

\smallskip \noindent
Wojciech Szyma{\'n}ski\\
Department of Mathematics and Computer Science \\
The University of Southern Denmark \\
Campusvej 55, DK-5230 Odense M, Denmark \\
E-mail: szymanski@imada.sdu.dk


\begin{thebibliography}{99}

\bibitem{B} J.-C. Birget,
{\it The groups of Richard Thompson and complexity},
Internat. J. Algebra Comput. {\bf 14} (2004), 569--626.

\bibitem{BE} O. Bratteli and D. E. Evans,
{\it Derivations tangential to compact groups: the non-abelian case},
Proc. London Math. Soc. (3) {\bf 52} (1986), 369--384.

\bibitem{CF} R. Conti and F. Fidaleo,
{\it Braided endomorphisms of Cuntz algebras},
Math. Scand. {\bf 87} (2000), 93--114.

\bibitem{CKS} R. Conti, J. Kimberley and W. Szyma{\'n}ski,
{\it More localized automorphisms of the Cuntz algebras},
arXiv:0808.2843, to appear in Proc. Edinburgh Math. Soc.

\bibitem{CP} R. Conti and C. Pinzari,
{\it Remarks on the index of endomorphisms of Cuntz algebras},
J. Funct. Anal. {\bf 142} (1996), 369--405.

\bibitem{CRS} R. Conti, M. R{\o}rdam and W. Szyma{\'n}ski,
{\it Endomorphisms of $\O_n$ which preserve the canonical UHF-subalgebra},
arXiv:0910.1304.

\bibitem{CS2} R. Conti and W. Szyma{\'n}ski,
{\it Computing the Jones index of quadratic permutation endomorphisms of ${\mathcal O}_2$},
J. Math. Phys. {\bf 50} (2009), 012705.

\bibitem{CS} R. Conti and W. Szyma{\'n}ski,
{\it Labeled trees and localized automorphisms of the Cuntz algebras},
arXiv:0805.4654, to appear in Trans. Amer. Math. Soc.

\bibitem{Cun1} J. Cuntz,
{\it Simple $C^*$-algebras generated by isometries},
Commun. Math. Phys. {\bf 57} (1977), 173--185.

\bibitem{Cun2} J. Cuntz,
{\it Automorphisms of certain simple $C^*$-algebras},
in {\it Quantum fields-algebras-processes}, ed. L. Streit, Springer 1980.

\bibitem{Cun3} J. Cuntz,
{\it Regular actions of Hopf algebras on the $C^*$-algebra generated by a Hilbert space},
in Operator algebras, mathematical physics and low dimensional topology (R. Herman, B.
Tanbay eds.), Res. Notes in Math. {\bf 5} (1993), 87--100.

\bibitem{DS} K. J. Dykema and D. Shlyakhtenko,
{\it Exactness of Cuntz-Pimsner $C^*$-algebras},
Proc. Edinburgh Math. Soc. (2) {\bf 44} (2001), 425--444.

\bibitem{H} J.-S. Huang,
Lectures on representation theory,
World Scientific Publ., River Edge, 1999.

\bibitem{I} M. Izumi,
{\it Subalgebras of infinite $C^*$-algebras with finite Watatani indices. I. Cuntz algebras},
Commun. Math. Phys. {\bf 155} (1993), 157--182.

\bibitem{K1} K. Kawamura,
{\it Polynomial endomorphisms of the Cuntz algebras arising from permutations. I. General theory},
Lett. Math. Phys. {\bf 71} (2005),  149--158.

\bibitem{K2} K. Kawamura,
{\it Branching laws for polynomial endomorphisms of Cuntz algebras arising from permutations},
Lett. Math. Phys. {\bf 77} (2006), 111--126.

\bibitem{Kit} B. P. Kitchens,
Symbolic dynamics. One-sided, two-sided and countable state Markov shifts,
Springer-Verlag, Berlin, 1998.

\bibitem{L2} R. Longo,
{\it Index of subfactors and statistics of quantum fields. II.
Correspondences, braid group statistics and Jones polynomial},
Commun. Math. Phys. {\bf 130} (1990), 285--309.

\bibitem{L} R. Longo,
{\it A duality for Hopf algebras and for subfactors. I},
Commun. Math. Phys. {\bf 159} (1994), 133--150.

\bibitem{Ne} V. Nekrashevych,
{\it Cuntz-Pimsner algebras of group actions},
J. Operator Theory {\bf 52} (2004), 223--249.


\bibitem{Price} G. Price,
{\it Extremal traces on some group-invariant C*-algebras},
J. Funct. Anal. {\bf 49} (1982), 145--151.

\bibitem{Sk2} A. Skalski,
{\it On automorphisms of $C^*$-algebras whose Voiculescu entropy is genuinely noncommutative},
arXiv:0911.3951.

\bibitem{Sk} A. Skalski,
{\it Noncommutative topological entropy of endomorphisms of Cuntz algebras II},
preprint, November 2009.

\bibitem{SZ} A. Skalski and J. Zacharias,
{\it Noncommutative topological entropy of endomorphisms of Cuntz algebras},
Lett. Math. Phys. {\bf 86} (2008), 115--134.

\bibitem{S} W. Szyma{\'n}ski,
{\it On localized automorphisms of the Cuntz algebras which preserve the diagonal subalgebra},
in `New Development of Operator Algebras', R.I.M.S. K\^{o}ky\^{u}roku {\bf 1587} (2008), 109--115.

\bibitem{W} Y. Watatani,
{\it Index for $C^*$-subalgebras},
Mem. Amer. Math. Soc. {\bf 83} (1990), no. 424.

\end{thebibliography}
\end{document}